\documentclass[12pt]{amsart}
\usepackage{amsfonts,amsmath,amssymb,amsthm,latexsym,verbatim}
\input xypic
\newtheorem{theorem}{Theorem}[section]
\newtheorem{lemma}[theorem]{Lemma}
\newtheorem{proposition}[theorem]{Proposition}

\theoremstyle{remark}
\newtheorem{remark}[theorem]{Remark}
\theoremstyle{definition}

\newtheorem{example}[theorem]{Example}

\DeclareMathOperator{\image}{{\mathrm{Im}}}

\DeclareMathOperator{\rad}{{\mathrm{rad}}}

\DeclareMathOperator{\End}{{\mathrm{End}}}
\newcommand{\abs}[1]{|#1|}	
\newcommand{\dprod}[2]{\langle {#1} , {#2}\rangle}	
\newcommand{\Ff}{{\mathcal {F}}} 

\newcommand{\F}{{\mathbf {F}}}

\DeclareMathOperator{\SL}{{\mathrm {SL}}} 	

\DeclareMathOperator{\rank}{{\mathrm {rank}}} 	
 	
\DeclareMathOperator{\Ch}{{\mathrm {Ch}}} 	
 	
\begin{document}
\title[Oppositeness modules]{Oppositeness in buildings 
and simple modules for finite groups of Lie type}
\author{Peter Sin}
\address{Department of Mathematics\\University of Florida\\ P. O. Box 118105\\ Gainesville
FL 32611\\ USA}
\date{}
\maketitle

\begin{abstract}
In the building of a finite group of Lie type we consider the incidence relations defined 
by oppositeness of flags. 
Such a relation gives rise to a homomorphism of permutation
modules (in the defining characteristic) whose image is a simple module for the group. 
The $p$-rank of the incidence relation is then the dimension of this simple module.
We give some general reductions towards the determination of the character of the 
simple module. Its highest weight is identified and the problem is reduced to the case of a prime field. The reduced problem can be approached through the representation theory of algebraic groups and the methods are illustrated for some examples.
\end{abstract}


\section{Introduction}

Let $G$ be a finite group with a split BN-pair of characteristic $p$ and rank $\ell$,
and let $I=\{1,\ldots,\ell\}$. The Weyl group $W$ is a euclidean reflection
group in a real vector space $V$, containing a root system $R$ and if
$S=\{\alpha_i\mid i\in I\}$ is a set of simple roots then $W$ is a Coxeter group
with generators the fundamental reflections $w_i$, $i\in I$, 
where $w_i$ is the reflection in the hyperplane perpendicular to the
simple root $\alpha_i$. 

For $J\subset I$, let $W_J:=\langle w_i \mid i\in J\rangle$ 
be the associated standard parabolic subgroup of $W$ 
and $P_J=BW_JB$ the standard
parabolic subgroup of $G$. 
By a {\it type} we simply mean a nonempty subset of $I$.
An {\it object of type} $I\setminus J$, or {\it of cotype} $J$ is, by definition, a right coset of $P_J$ in $G$. 

The associated building is the simplicial complex in which the simplices
are the cosets of the standard parabolic subgroups of $G$,
and the face relation is the reverse of inclusion. An object
of type $I\setminus J$ is a $(\abs{I\setminus J}-1)$-simplex.

An example which has an elementary description is the
building for $G=\SL(V)$ of type $A_\ell$, where $V$ is an $(\ell+1)$-dimensional
vector space. A $k$-simplex is a flag  of $k+1$ nontrivial,  proper subspaces
of $V$, and its type is the set of dimensions of the subspaces in the flag.

Given two fixed types there are various  incidence relations
between objects of the respective types which are preserved by the action of $G$. 
In the above example, if we consider in $V$ the complete flags (objects of type $I$)
and the $i$-dimensional subspaces (objects of type $\{i\}$), 
then an incidence relation can be specified
by prescribing a sequence of $\ell$ dimensions for the intersections
of an $i$-dimensional subspace with the subspaces in a complete flag.

An incidence relation $\mathcal R$ between the sets of objects of cotypes $J$ and $K$
can be encoded by an incidence matrix or, more conveniently for our purposes,
by an {\it incidence map of permutation modules}.
Let $k$ be a commutative ring with 1.
Let $\Ff_J$ denote the space of functions from the set $P_J\backslash G$ of objects of cotype $J$ to $k$. Then $\Ff_J$ is a left $kG$-module by the rule
\begin{equation}\label{gaction}
(xf)(P_Jg)=f(P_Jgx), \quad f\in \Ff_J,\quad g,x\in G.
\end{equation}
Let $\delta_{P_Jg}$ denote the characteristic function of the object $P_Jg\in P_J\backslash G$.  These characteristic functions are permuted transitively by $G$ and form a
basis for $\Ff_J$.

A $G$-invariant relation $\mathcal R$ defines a $kG$-homomorphism
$\eta: \Ff_J\to\Ff_K$ given by
\begin{equation}
 \eta(f)(P_Kh)=\sum_{P_Jg\mathcal R P_Kh}f(P_Jg).
\end{equation}
The characteristic function of an object of cotype $J$ is sent
to the sum of the characteristic functions of all objects  of cotype $K$ which
are incident with it. 

Naturally, we would like to compute invariants of $\eta$.
We can ask for its Smith Normal Form when $k=\mathbb Z$, 
its rank when $k$ is a field, or its eigenvalues if $J=K$.
For greater detail we can also consider the $kG$-module structure of the image of $\eta$.
We have described the problem in great generality and,
as  might be expected, the nature and difficulty of a specific instance will 
depend very much on $G$, $\mathcal R$ and $k$. 
For example, if $k$ is a field then whether or
not its characteristic is the same as that of $G$ is crucial, because 
the representation theory of $G$ in its defining characteristic is closely
related to the representation theory of reductive algebraic groups, while
the cross-characteristic representation theory of $G$ has a closer connection
to its complex character theory and to Brauer's theory of blocks. 
This dichotomy can be observed in the case of generalized quadrangles
by comparing the results in \cite{BBW}, \cite{SaS} and \cite{CSX}. 
In the first two papers $k$ has characteristic 2, while $G$ 
has odd characteristic in the first and
even characteristic in the second. In the third, the group  $G$ and
$k$ have the same odd characteristic. 

In this paper we restrict ourselves to the case where $k$ is a field of the same
characteristic as $G$ and turn next to the choice of $\mathcal R$.
We know from examples (e.g. \cite{BaS}) that the permutation $kG$-module $\Ff_J$
can have many composition factors, growing 
with respect to the Lie rank of $G$, and also growing for fixed rank
and fixed characteristic as the size of the field increases.
The same may hold for the image of $\eta$. Even for rank $2$,  examples 
(e.g. \cite{SaS}) demonstrate that the number of composition factors of
$\image\eta$ can grow without bound. In many cases, one knows neither
the number of composition factors nor their dimensions, and 
in these cases we have no formula for $\rank \eta$. 
There are also examples of two problems which seem very similar on the surface,
but one turns out to be much harder than the other. 
For example, the problem of computing the rank of the inclusion relation 
of $2$-dimensional subspaces in $(n-2)$-dimensional subspaces in a vector space of dimension $n\geq 5$ is unsolved. In contrast, a simple formula is known (\cite{Sin}) 
for the rank of the relation of nonzero intersection.
It is therefore desirable, for a given $\mathcal R$, to have some idea 
{\it a priori} of the complexity of the $kG$-submodule lattice of $\image\eta$.
From this point of view, the oppositeness relations (defined below),
which will be shown to give simple modules, can be considered to form the starting 
point of the theory.

\section{Oppositeness in buildings}
Let $R^+$ denote the set of positive roots.
The length of an element $w\in W$, denoted $\ell(w)$, is the length of the shortest expression for $w$ as a word in the generators $w_i$. This is also equal to the number
of positive roots which $w$ transforms to negative roots.
There is a unique element of maximal length, denoted $w_0$, which sends all
positive roots to negative roots.

Notions of oppositeness exist at the level of types and at the level of objects.
Two types $J$ and $K$ are {\it opposite} if $\{-w_0(\alpha_i) \mid i\in J\}=\{\alpha_j\mid j\in K\}$, or, equivalently, if $\{w_0w_iw_0\mid i\in J\}= \{w_i\mid i\in K\}$.
In other words, two types are opposite when they are mapped to each other
by the symmetry of the Dynkin diagram induced by $w_0$.
If $w_0=-1$, then every type is opposite itself;
this holds for all connected root systems
except for those of type $A_\ell$, $D_\ell$  ($\ell$ odd) and $E_6$.
Let $J$ and $K$ be fixed opposite cotypes. 
An object $P_Jg$ of cotype $J$ and an object $P_Kh$ of cotype $K$
are  {\it opposite} each other 
if $P_Khg^{-1}P_J=P_Kw_0P_J$ ($\iff P_Kh\subseteq P_Jw_0P_Kg
\iff P_Jg\subseteq P_Kw_0P_Jh$). 

The definition of oppositeness  is a precise way to formulate the
intuitive idea of two objects being in ``general position''.
From now on we shall let $J$ and $K$ be fixed opposite cotypes. 

\begin{example} Consider a universal Chevalley group $G$ type $A_\ell$ and $J=I\setminus\{i\}$.
Then $G\cong\SL(V)$ for an $(\ell+1)$-dimensional vector space $V$.
The objects of type $\{i\}$ can be identified with $i$-dimensional subspaces
of $V$ and objects of the opposite type
$\{\ell+1-i\}$ can  be identified with $\ell+1-i$-dimensional
subspaces. A subspace of  type $\{i\}$ is opposite one of type 
$\{\ell+1-i\}$ if their
intersection is the zero subspace.  A familar special
case is when $\ell=3$ and $i=\ell+1-i=2$. If we think projectively, the objects are
lines in space and the oppositeness relation is skewness.

More generally, an object of cotype $J=\{j_1,\ldots, j_m\}$ is  a flag 
$$
V_{j_1}\subset V_{j_2}\subset\cdots\subset V_{j_m}
$$
of subspaces of $V$ with $\dim V_{i_j}=i_j$.
If $V'_{k_1}\supset V'_{k_2}\supset\cdots\supset V'_{k_m}$ is an object of the opposite type, then the two flags are opposite iff $V_{i_j}\cap V'_{k_j}=\{0\}$, for $j=1$, \dots, $m$.
\end{example}

\begin{example}
Let $G$ be of type $B_\ell$ or $C_\ell$ with
$\ell\geq 2$ or $D_\ell$ with $\ell\geq 3$  and let  $J=I\setminus\{1\}$. 
Then  $J$ is opposite to itself. In the $B_\ell$ case, objects of cotype $J$
can be identified with singular points (one-dimensional subspaces) 
with respect to a nondegenerate quadratic form in a finite vector space of 
dimension $2\ell+1$. Two 
singular points are opposite if and only if they are not orthogonal.
Similarly, objects of cotype $J$ can be viewed as
singular points of a $2\ell$-dimensional vector space 
with respect to a symplectic symplectic form for type $C_\ell$ or a 
quadratic form for type $D_\ell$, with oppositeness interpreted as non-orthogonality.
Two points are opposite if and only if they do not lie on a singular line.
Thus, the concept of oppositeness generalizes the concept of
collinearity of singular points in these classical geometries.   
\end{example}

\begin{example}\label{E6example}
Consider a universal Chevalley group $G$ of type $E_6(q)$. 
This group has a concrete description as the group of linear transformations
which preserve a certain cubic form on a 27-dimensional vector space $V$ over $\F_q$.
The geometry of this space has been studied in great detail. 
(See \cite{Asch}, \cite{CohCoop}, \cite{Cooperstein}.)
Consider the objects of type $1$ and the opposite type $6$. (See Figure 1.)
We can  view these,
respectively, as the singular points and singular (in a dual sense) hyperplanes of  $V$.
A singular point $\langle v\rangle$ is opposite a singular
hyperplane $H$ if and only $v\notin H$. 

\begin{figure}[h]
\centering
\begin{picture}(100,100)(50,-50)
\thicklines
   \put(5,30){$\alpha_1$}
   \put(55,30){$\alpha_2$}
   \put(105,30){$\alpha_3$}
   \put(155,30){$\alpha_5$}
   \put(205,30){$\alpha_6$}
   \put(105,-45){$\alpha_4$}
   \put(10,20){\line(1,0){200}}
  \put(110,20){\line(0,-1){50}}
  \multiput(10,20)(50,0){5}{\circle*{5}}
  \put(110,-30){\circle*{5}}
  \put(210,20){\circle*{10}}
  \put(10,20){\circle{10}}
\end{picture}
\caption{}
\end{figure}
\end{example}
For further examples of oppositeness, we refer the reader to \cite{Brouwer}.

Let $A=A_{J,K}$ be the incidence matrix of the oppositeness relation,
with rows indexed by objects of cotype $J$ (in some order)
and columns indexed by objects of cotype $K$.
Suppose $G$ is  defined over $\F_q$, where $q$ is a power of $p$. 
Brouwer \cite{Brouwer} has shown that each eigenvalue
of $AA'$  is a power of $q$, where $A'$ is the transpose of $A$.

In the following sections, we show that the $p$-rank of the incidence matrix $A$
is the dimension of an irreducible $p$-modular
representation of $G$. This fact is derived as a corollary of a general theorem of  
Carter and Lusztig \cite{CarterLusztig}. 
Then we describe the simple module in terms of its highest weight
and show, using Steinberg's Tensor Product Theorem that, given a  root system and choice
of opposite types, the $p$-ranks can be computed for all $q$
once they are known in the case $q=p$. 
We then discuss methods for computing the character of the
simple module in some examples.

\section{Some lemmas on double cosets}
Let $V_J$ be the subspace of $V$ spanned by $S_J=\{\alpha_i\mid i\in J\}$.
Then $R_J=R\cap V_J$ is a root system in $V_J$ with simple system $S_J$
and Weyl group $W_J$. For $w\in W_J$, its length as an element of $W_J$ is the same as 
for $W$. Let $w_J$ be the longest element in $W_J$.

The following is immediate from the definition of opposite types.
\begin{lemma}
If $w_0=w^*w_J=w_Kw'$ then $w^*=w'$.
\end{lemma}
\qed

For $w\in W$, we recall that $U^-_w=U\cap w^{-1}w_0Uw_0w$. 
For $w\neq 1$ this group is a nontrivial $p$-group. We also
have that given a choice of preimage $n_w\in N$ of $w$,
each element of $BwB$ can be written as a unique product
$g=bn_wu$, with $b\in B$ and $u\in U^-_w$. Thus,
$\abs{U^-_w}$ equals the number of cosets $Bg'$ in $BwB$, which is equal
to $q^{\ell(w)}$ in the untwisted case. In the case of twisted groups,
the equation  $\abs{U^-_w}=q^{\ell(w)}$ is also valid if
$q$ is taken to be the {\it level} \cite[Definition 2.1.9]{GLS3} 
of the Frobenius endomorphism defining $G$ 
(which is a prime power or, for Suzuki and Ree groups, the
square root of an odd power of $2$ or $3$) 
and $\ell(w)$ means the length in the Weyl group of 
the untwisted root system. We refer to \cite[Theorems 2.3.8 and 2.4.1]{GLS3}
and \cite{CarterBook}. 
We shall use only the obvious fact that $U^-_w$ is a nontrivial $p$-group when
$w\neq 1$.)

\begin{lemma}(\cite[Proposition 3.1]{Brouwer}) We have $P_Kw_0P_J=P_Kw^*P_J=P_Kw^*B$.
The number of cosets $P_Kg$ in $P_Kw_0P_J$ is $q^{\ell(w^*)}$.
\end{lemma}
\qed

\begin{lemma}\label{countfiber}(cf. \cite[Corollary 3.2]{Brouwer})
Let $w\in W_J$. Then $B(w^*w)B\subseteq P_Kw_0P_J$. Let $x\in G$.
For a given coset
$P_Kh\subseteq P_Kw_0P_Jx$,  the number of cosets $Bg\subseteq Bw^*wBx$ such that
$Bg\subseteq P_Kh$ is $q^{\ell(w)}$.
\end{lemma}
\begin{proof}
The first assertion is true because $P_Kw^*wP_J=P_Kw^*P_J=P_Kw_0P_J$.
The rest follows by counting.
The number of cosets $P_Kh\subseteq P_Kw_0P_Jg$ is $q^{\ell(w^*)}$
and the number of cosets $Bg\subseteq Bw^*wBx$ is 
$q^{\ell(w^*w)}=q^{\ell(w^*)+\ell(w)}$.
\end{proof}

\section{Permutation modules on flags and their oppositeness homomorphisms}
Let $k$ be an algebraically closed  field of characteristic $p$.
Recall from the Introduction that $\Ff_J$ denotes the space of functions from the set $P_J\backslash G$ of objects of cotype $J$ to $k$,  with left $kG$-module structure given by 
(\ref{gaction}).
Also, $\delta_{P_Jg}$ denotes the characteristic function of the object $P_Jg\in P_J\backslash G$ and these functions form a permutation basis of $\Ff_J$.
(The module $\Ff_J$ is isomorphic to the $kG$-permutation module on the {\emph left}
cosets of $P_J$ in $G$ by the map sending $\delta_{P_Jg}$ to $g^{-1}P_J$.)
The relation of oppositeness defines the $kG$-homomorphism
$\eta: \Ff_J\to\Ff_K$ given by
\begin{equation}
 \eta(f)(P_Kh)=\sum_{P_Jg\subseteq P_Jw_0P_Kh}f(P_Jg)
\end{equation}
and we have
$$
\eta(\delta_{P_Jg})=\sum_{P_Kh\subseteq P_Kw_0PJg}\delta_{P_Kh}.
$$

The following result is essentially
a corollary of more general results in \cite{CarterLusztig}.

\begin{theorem}\label{main}
The image of $\eta$ is a simple module, uniquely characterized by
the property that its one-dimensional $U$-invariant subspace has full stablizer equal to $P_J$, which acts trivially on it.
\end{theorem}

The next subsections  describe some results in \cite{CarterLusztig}
and explain how Theorem~\ref{main} follows from them.

\subsection*{Fundamental endomorphisms of ${\Ff}_\emptyset$}

Let $\Ff={\Ff}_\emptyset$. 
For each $w\in W$ we define $T_w\in\End_k(\Ff)$ by the formula 
$$
T_w(f)(Bg)=\sum_{Bg'\subseteq Bw^{-1}Bg}f(Bg').
$$
Then 
$$
T_w\in \End_{kG}(\Ff),\quad\text{for all $w\in W$}.
$$
One can show (see \cite{CarterLusztig}) that
$$
T_{ww'}=T_wT_{w'} \quad \text{if $\ell(ww')=\ell(w)+\ell(w')$.}
$$

Let $w\in W$ have reduced expression
$$
w_{j_n}\cdots w_{j_1}.
$$
We consider the partial products $w_{j_1}$, $w_{j_2}w_{j_1}$, \dots 
$w_{j_n}\cdots w_{j_1}$. The positive roots sent to negative roots
by each partial product are also sent to negative roots by each of its
successors and each partial product sends exactly one more positive
root to a negative root than its predecessors. Explicitly, the new positive
root sent by $w_{j_i}\cdots w_{j_1}$ to a negative root is
$w_{j_1}\cdots w_{j_{i-1}}(r_{j_i})$.
With this in mind, we can define the endomorphsim $\Theta^J_{w_0}$
of  \cite[p.363]{CarterLusztig} for any $J$ subset of $I$.

For any  reduced expression 
$$
w_0=w_{j_k}\cdots w_{j_1}
$$
define
\begin{equation}
\Theta_{j_i}= 
\begin{cases} T_{w_{j_i}}\quad\text{if $w_{j_1}\cdots w_{j_{i-1}}(r_{j_i})\notin V_J$}\\
              I+T_{w_{j_i}}\quad\text{if $w_{j_1}\cdots w_{j_{i-1}}(r_{j_i})\in V_J$}
\end{cases}
\end{equation}
and set 
$$
\Theta^J_{w_0}=\Theta_{j_k}\Theta_{j_{k-1}}\cdots\Theta_{j_1}.
$$

The definition depends on the choice of reduced expression but it is shown
in \cite{CarterLusztig} that different expressions give the same endomorphism
up to a nonzero scalar multiple.

We can now state the result from which Theorem~\ref{main} can be deduced.
We recall that in every simple $kG$-module the subspace fixed by the subgroup
$U$ is one-dimensional. Therefore, the full stabilizer in $G$ of this subspace
must contain $N_G(U)=B$, so must be equal to some standard parabolic subgroup $P_Q$,
$Q\subseteq I$. Thus this line is a one-dimensional $kP_Q$-module.

\begin{theorem}(\cite[Theorem 7.1]{CarterLusztig})
The image $\Theta^J_{w_0}(\Ff)$ is a simple $kG$-module. The
full stablizer of the one-dimensional subspace of $U$-fixed points
in this module is $P_J$ and the action of $P_J$  on this one-dimensional
subspace is trivial.
\end{theorem}

(Our space $\Ff$ is the space denoted $\Ff_\chi$ in \cite{CarterLusztig} when
 $\chi$ is the trivial character.)

\subsection*{Proof of Theorem~\ref{main}}
Let
\begin{equation}\label{wJ}
w_J=w_{i_m}\cdots w_{i_2}w_{i_1}
\end{equation}
be a reduced expression for $w_J$. The above expression can be extended to a reduced expresion
\begin{equation}\label{w0}
w_0=w_{i_k}\cdots w_{i_{m+1}}w_{i_m}\cdots w_{i_1}
\end{equation}
of $w_0$. Thus $m=\abs{R_J^+}$ and $k=\abs{R^+}$.

Then 
\begin{equation}\label{w*}
w^*=w_{i_k}\cdots w_{i_{m+1}}.
\end{equation}
is a reduced expression for $w^*$.
We choose the special expression (\ref{w0}) for $w_0$ to define $\Theta^J_{w_0}$.
Since $w_J$ sends all positive roots 
in $V_J$ to negative roots and $w_0$ sends all negative roots to positive roots, 
it is clear
that for the first $m$ partial products the new positive root sent to a negative
root belongs to $V_J$, and that the new positive roots for the remaining partial products 
are the elements of $R^+\setminus R_J^+$, so do not belong to $V_J$.
Thus we have
\begin{equation}\label{Theta}
\Theta^J_{w_0}=T_{w^*}(1+T_{i_m})\cdots(1+T_{i_1}),
\end{equation}

Since $\ell(w^*w)=\ell(w^*)+\ell(w)$ for all $w\in W_J$, we see that
$\Theta^J_{w_0}$ is a sum of endomorphisms of the form
$T_{w^*w}$, for certain elements $w\in W_J$, with exactly one term of this sum
equal to $T_{w^*}$.

Let $\pi_J:\Ff\to\Ff_J$ be defined by
$$
(\pi_J(f))(P_Jg)=\sum_{Bh\subseteq P_Jg}f(Bh)
$$
and $\pi_K$ defined similarly. It is easily checked that
$\pi_J$ and $\pi_K$ are $kG$-module homomorphisms and they are surjective
since $\pi_J(\delta_{B})=\delta_{P_J}$.

We can now complete the proof of Theorem~\ref{main}.
The main step is to compare $\eta\pi_J$ with $\pi_KT_{w^*w}$ for  $w\in W_J$.
For $f\in\Ff$, we compute 
\begin{equation}
\begin{aligned}
\left[\eta(\pi_J(f))\right](P_Kg)&=\sum_{P_Jh\subseteq P_J{w^*}^{-1}P_Kg}\sum_{Bx\subseteq P_Jh}f(Bx)\\
                      &=\sum_{Bx\subseteq P_J{w^*}^{-1}P_Kg}f(Bx).
\end{aligned}
\end{equation}
and
\begin{equation}
\begin{aligned}
\left[\pi_K(T_{w^*w}(f))\right](P_Kh)&=\sum_{Bg\subseteq P_Kh}(T_{w^*w}f)(Bg)\\
&=\sum_{Bg\subseteq P_Kh}\sum_{Bx\subseteq B(w^*w)^{-1}Bg}f(Bx)\\
&=\sum_{Bg\subseteq P_Kh}\sum_{Bg\subseteq B(w^*w)Bx}f(Bx)\\
                    &=q^{\ell(w)}\sum_{Bx\subseteq P_J{w^*}^{-1}P_Kg}f(Bx).
\end{aligned}
\end{equation}
where the last equality follows from Lemma~\ref{countfiber}, since
$Bx\subseteq B(w^*w)^{-1}Bg$ if and only if $Bg\subseteq Bw^*wBx$.
Thus, we have for each $w\in W_J$ a commutative diagram
\begin{equation}
\xymatrix{
\Ff_J\ar[rr]^{q^{\ell(w)}\eta}&&\Ff_K\\ 
\Ff\ar[u]^{\pi_J}\ar[rr]^{T_{ww^*}} & & \Ff\ar[u]^{\pi_K},
}
\end{equation}
which means that for $w\neq 1$ we have $\pi_KT_{ww^*}=0$.
It now follows from (\ref{Theta}) that $\pi_K\Theta^J_{w_0}=\pi_KT_{w^*}=\eta\pi_J$.
Therefore, since $\Theta^J_{w_0}(\Ff)$ is simple and 
$\eta\pi_J(\Ff)\neq 0$, we see that 
$\eta\pi_J(\Ff)\cong \Theta^J_{w_0}(\Ff)$.
Finally,  since $\pi_J$ is surjective, we have $\eta(\Ff_J)\cong \Theta^J_{w_0}(\Ff)$.
This completes the proof of Theorem~\ref{main}.
\qed

\section{Highest weights}\label{hw}

If $G$ is a universal Chevalley group  or a twisted subgroup
of such a group then there is an algebraic group  $G(k)$,
the Chevalley group over $k$, and a Frobenius endomorphism $\sigma$ of $G(k)$
such that $G$ is the group of fixed points of $\sigma$. 
References for this theory are the original paper \cite[12.4]{St2} or the very complete 
account in \cite[Chapter 2]{GLS3}. 
 
The simple $kG$-modules are restrictions of certain
simple rational modules $L(\lambda)$ of $G(k)$ (\cite[13.1]{St2}),
where $\lambda$ is the highest weight of the module. This connection is
well known; for definitions and details we refer to the original sources \cite{St} and \cite{St2}.
The simple module $L(\lambda)$ is characterized by the property
that it has a unique $B$-stable line, and  $T$ acts on this line
by the character $\lambda$.

Assume that $G$ is a universal Chevalley group over $\F_q$. 
By Steinberg's theorem \cite[13.3]{St2}, the simple $G(k)$ modules with
highest weights in the set of \emph{$q$-restricted weights}
\begin{equation*}
X_q=\{\lambda=\sum_{i=1}^\ell a_i\omega_i\in X_+ \mid \,  0\leq a_i\leq q-1\ (\forall i)\}
\end{equation*}
remain irreducible upon restriction to $G$ and this gives a complete set
of mutually non-isomorphic simple $kG$-modules.

It is useful to identify the highest weight of the module in Theorem~\ref{main}.
We may assume that $R$ is indecomposable, since simple modules
for direct products are easily described in terms of the factors.

Let $\lambda_{opp}$ denote the $q$-restricted highest weight
such that the restriction of the simple $G(k)$-module
$L(\lambda_{opp})$ to $G$ is the simple module  in Theorem~\ref{main}.
The condition in Theorem~\ref{main} that $P_J$ is the full
stabilizer in $G$ of the one-dimensional highest weight space  of
$L(\lambda_{opp})$ is equivalent to the condition that
for $i=1$,\dots, $\ell$,
$\langle \lambda_{opp}, \alpha_i^\vee\rangle=0$ if and only if $i\in J$.
The condition in Theorem~\ref{main} that $P_J$
acts trivially on the highest weight space means that the
restriction of $\lambda_{opp}$ to $H$ is the trivial character,
which in turn means that, when $\lambda_{opp}$
is written as a linear combination of fundamental weights,
all of the coefficients must be $0$ or $q-1$.
These conditions allow us to identify $\lambda_{opp}$.
The fundamental weights $\omega_i$ for the ambient algebraic group
are indexed by $I$, and $\lambda_{opp}=\sum_{i\in I\setminus J}(q-1)\omega_i$.

Consider next a twisted group $G$, constructed from 
an overlying universal Chevalley  group $G^*$ over $\F_q$, as the group
of fixed points of an automorphism induced by an isometry 
$\rho$ of order $e$ of the Dynkin diagram of $G^*$.
Then $q=q_0^e$ for some prime power $q_0$. 
Let $I^*=\{1,\ldots,\ell^*\}$ index the fundamental roots for $G^*$. Then
the index set $I$ for $G$ labels the $\rho$-orbits on $I^*$.
Let $\omega_i$, $i\in I^*$ be the fundamental weights of the ambient algebraic group.
For $J\subseteq I$, let $J^*\subset I^*$ be the union of the orbits in $J$. 
Then $\lambda_{opp}=\sum_{i\in I^*\setminus J^*}(q_0-1)\omega_i$.

Finally, we consider the case of Suzuki and Ree groups.
Here $G$ is the subgroup of fixed points in $G(k)$ of
a Steinberg endomorphism $\tau$ which induces a length-changing permutation 
of the Dynkin diagram of $G(k)$ and such that $\tau^2=\sigma$ is a Frobenius
endomorphism of $G(k)$ with respect to a rational structure over a finite field
$\F_q$. Then the set $I$ for $G$ indexes the subset of fundamental
weights of the ambient algebraic group which are orthogonal to the long simple roots, and for $J\subset I$, we have
$\lambda_{opp}=\sum_{i\in I\setminus J}(q-1)\omega_i$.

\begin{example}
As examples, we can consider the extreme cases.
If  $J=K=I$, then $L(\lambda_{opp})\cong k$.
If $J=K=\emptyset$,  $L(\lambda_{opp})$ is the {\it Steinberg module} of the group $G$, which has dimension equal to the the order of $U$, the $p$-sylow subgroup of $G$.
\end{example}

From the discussion above, we see that 
$\lambda_{opp}$  has the form $(q-1)\tilde\omega$,
where $\tilde\omega$ is a sum of fundamental weights. 
If $q=p^t$, then by  Steinberg's Tensor Product Theorem, we have an isomorphism

\begin{equation}
L((q-1)\tilde\omega)\cong L((p-1)\tilde\omega)\otimes L((p-1)\tilde\omega)^{(p)}\otimes\cdots
\otimes L((p-1)\tilde\omega)^{(p^{t-1})}
\end{equation}
as modules for the algebraic group. Here, the superscripts indicate
twisting by powers of the Frobenius morphism. This twisting changes the isomorphism 
type of a module, but does not change its dimension.
In particular, we have the following numerical result.

\begin{proposition}
Let the root system $R$ and opposite cotypes $J$ and $K$ be given and let
 $A(q)=A(q)_{J,K}$ denote the 
oppositeness incidence matrix  for objects of cotypes $J$ and $K$ 
in the building over $\F_q$, where $q=p^t$.
Then $\rank_pA(q)=(\rank_p A(p))^t$. 
\end{proposition}

\begin{remark}
The proposition states that once $\rank_p A(p)$ is known then we know $\rank_p A(q)$ for all powers $q$ of $p$. This reduction of the computation to the prime case is
significant from the viewpoint of representation theory of algebraic groups, where the Weyl modules  with modules highest weight $(p-1)\tilde\omega$ are much less 
complex in structure than those of highest weight $(p^2-1)\tilde\omega$, say.

A very useful tool for analyzing Weyl modules is the {\it Jantzen Sum Formula} \cite[II.8.19]{Jantzen}:
The Weyl module $V(\lambda)$ has a descending filtration,
of submodules $V(\lambda)^i$, $i>0$, such that
\begin{equation*}
V(\lambda)^1=\rad V(\lambda),\quad\text{(so that $V(\lambda)/V(\lambda)^1\cong L(\lambda)$)}
\end{equation*}
and 
\begin{equation}\label{jsum}
\sum_{i>0}\Ch( V(\lambda)^i)= -\sum_{\alpha>0}
\sum_{\{m: 0<mp<\dprod{\lambda+\rho}{\alpha^\vee}\}}v_p(mp)\chi(\lambda-mp\alpha)
\end{equation}
First, we recall the notation in this formula (which is standard, and follows 
from \cite{Jantzen}). The module $V(\lambda)$ is the Weyl module and the module $L(\lambda)$ is
its simple quotient. The weight $\rho$ is the half-sum of the positive roots
and $v_p(m)$ denotes the exponent of $p$ in the prime factorization of $m$.
Finally, the character $\chi(\mu)$ is the so-called Weyl character; there is a 
unique weight of the form $\mu'=w(\mu+\rho)-\rho$ in the region $\{\nu :\langle \nu+\rho,\alpha^\vee\rangle\geq 0, \forall \alpha\in R^+\}$, where $w\in W$.
Then $\chi(\mu)$ is the $\text{sign}(w)\Ch V(\mu')$ if $\mu'$ is dominant, and zero
otherwise. 
The usefulness of the sum formula is based on the fact that
the characters of the Weyl modules themselves are given by Weyl's Character Formula,
so that the right hand side can be computed from  $p$, $R$ and $\lambda$.
In \cite{Ars} there is a  detailed description of a procedure  
for performing this computation. 
We shall refer to the quantity in (\ref{jsum}) as the Jantzen sum for $V(\lambda)$. 
As can be seen from the left hand side
of (\ref{jsum}), the Jantzen sum gives an upper estimate on the composition multiplicities in the radical of the Weyl module $V(\lambda)$ in terms of the composition
factors of Weyl modules which have lower highest weights. Sometimes, for weights of
a special form, more information can be obtained, using induction and
other facts from representation theory. For example when
we start with a weight of the form $\lambda=(p-1)\tilde\omega$, 
it may be that the highest weights of the Weyl characters $\chi(\mu)$ 
in the Jantzen sum are very few in number
or all have a similar form, such as $r\tilde\omega$ for $r<p-1$. 
In such cases, it is possible to deduce the character of $L((p-1)\tilde\omega)$.
Numerous examples of this method were worked out in detail in \cite{Ars}. 
\end{remark}

\section{Examples}
We consider again some of the examples from the Introduction.
\begin{example}
For type $A_\ell$, $J=I\setminus\{i\}$, the $p$-ranks have been computed
in \cite{Sin}. (In fact it is the $p$-rank of the complementary relation
of nonzero intersection which is computed, which equals one plus the $p$-rank
for zero intersection.)
In this case, the simple modules $L((p-1)\omega_i)$ can be found
without reference to Weyl modules. Let $S(i(p-1))$ denote the
degree $i(p-1)$ homogeneous component of the truncated 
polynomial ring $k[x_0,\ldots,x_\ell]/(x_i^p; 0\leq i\leq\ell)$, with
$G\cong\SL((\ell+1,k)$ acting through linear substitutions. Then it is well known and 
elementary to show that $S(i(p-1))$ is a simple $kG$-module.
By inspecting the highest weights, we see that $S(i(p-1))\cong L((p-1)\omega_{\ell+1-i})$,
for $i=1$,\dots, $\ell$.
\end{example}

\begin{example}
For the examples of types $B_\ell$, $C_\ell$ and $D_\ell$, with $J=I\setminus\{1\}$
concerning singular points in classical spaces, the $p$-ranks have been
computed in \cite{Ars} by analysis of the  Weyl modules, as outlined above.
The Weyl modules in question are $V((p-1)\omega_1)$ and for type $C_\ell$ they
are simple, so work is only needed for types $B_\ell$ and $D_\ell$.
The method can also be extended to the classical modules of the classical
groups of twisted type, namely the non-split orthogonal groups (type $^2D_\ell$)
and the unitary groups (type $^2A_\ell$). In the latter case
one must study the Weyl module $V((p-1)(\omega_1+\omega_\ell))$, in accordance
with our discussion  in \S\ref{hw}.
\end{example}

\begin{example}
We will sketch the computation of the $p$-rank for Example~\ref{E6example}.
The Jantzen sum for $V((p-1)\omega_1)$ is equal to the following:
\begin{multline}\label{alt}
\chi((p-7)\omega_1+3\omega_6)-\chi((p-8)\omega_1+\omega_4+2\omega_6)+
\chi((p-9)\omega_1+\omega_3+\omega_6)\\
-\chi((p-10)\omega_1+\omega_2+\omega_5)+\chi((p-11)\omega_1+2\omega_2)
\end{multline}
The next step is to study the structure of the Weyl modules appearing in this
character, again by using the sum formula.

We observe first that the Jantzen sum  for 
$V((p-11)\omega_1+2\omega_2)$ is zero, so this Weyl module is simple.

Next, the Jantzen sum for $V((p-10)\omega_1+\omega_2+\omega_5)$
is $\chi((p-11)\omega_1+2\omega_2)$. It follows that the radical of
$V((p-10)\omega_1+\omega_2+\omega_5)$ is isomorphic to the simple module
$V((p-11)\omega_1+2\omega_2)$.

The Jantzen sum for 
$V((p-9)\omega_1+\omega_3+\omega_6)$ is equal to 
$$
\chi((p-10)\omega_1+\omega_2+\omega_5)-\chi((p-11)\omega_1+2\omega_2),
$$
which by the 
previous paragraph is equal to $\Ch L((p-10)\omega_1+\omega_2+\omega_5)$. Therefore
$\rad V((p-9)\omega_1+\omega_3+\omega_6)\cong L((p-10)\omega_1+\omega_2+\omega_5)$.

The Jantzen sum for $V((p-8)\omega_1+\omega_4+2\omega_6)$
is 
$$
\chi((p-9)\omega_1+\omega_3+\omega_6)-\chi((p-10)\omega_1+\omega_2+\omega_5)+
\chi((p-11)\omega_1+2\omega_2)=\Ch L((p-9)\omega_1+\omega_3+\omega_6),
$$
which
leads us to conclude that 
$\rad V((p-8)\omega_1+\omega_4+2\omega_6)\cong L((p-9)\omega_1+\omega_3+\omega_6$.

The Jantzen sum for $V((p-7)\omega_1+3\omega_6)$
is 
\begin{multline}
\chi((p-8)\omega_1+\omega_4+2\omega_6)-
\chi((p-9)\omega_1+\omega_3+\omega_6)+\chi((p-10)\omega_1+\omega_2+\omega_5)\\
-\chi((p-11)\omega_1+2\omega_2)\cong \Ch L((p-8)\omega_1+\omega_4+2\omega_6).
\end{multline}
Hence $\rad V((p-7)\omega_1+3\omega_6)\cong L((p-8)\omega_1+\omega_4+2\omega_6)$.

Finally, by (\ref{alt}), we have $\rad V((p-1)\omega_1)\cong L((p-7)\omega_1+3\omega_6)$.
Therefore, there is an exact sequence
\begin{multline*}
0\to V((p-11)\omega_1+2\omega_2)\to V((p-10)\omega_1+\omega_2+\omega_5)\\
\to V((p-9)\omega_1+\omega_3+\omega_6)\to V((p-8)\omega_1+\omega_4+2\omega_6)\\
\to V((p-7)\omega_1+3\omega_6)\to V((p-1)\omega_1)\to L((p-1)\omega_1)\to 0
\end{multline*}
The  dimensions of the $V(\mu)$ are given by Weyl's formula. 
Hence,
\begin{multline*}
\dim L((p-1)\omega_1)=\frac{1}{2^7.3.5.11}p(p+1)(p+3)\\
\times(3p^8-12p^7+39p^6+320p^5\\
-550p^4+1240p^3+2080p^2-1920p+1440)
\end{multline*}
\end{example}

\section{Acknowledgements}
Thanks are due to the anonymous referees for their careful reading and
helpful suggestions.


\begin{thebibliography}{99}
\bibitem{Asch} Aschbacher, Michael, The $27$-dimensional module for $E_6$, I,
Inventiones Mathematicae {\bf 89} (1987), 159--195

\bibitem{Ars} Arslan, Og\"ul and Sin, Peter, Some simple modules for classical groups and p-ranks of orthogonal and Hermitian geometries, {\it J. Algebra} {\bf 327} (2011),
141--169.

\bibitem{BBW} B.Bagchi and A.E.Brouwer and H.A.Wilbrink,
 Notes on binary codes related to the $O(5,q)$ generalized 
 quadrangle for odd $q$, Geometriae Dedicata {\bf 39} (1991), 339--355.

\bibitem{BaS} M. Bardoe, P. Sin, The permutation modules for
$\mathrm{GL}(n+1,\mathbb{F}_q)$ acting on $\mathbb{P}^n(\mathbb{F}_q)$ and $\mathbb{F}_q^{n+1}$, {\it J. London Math. Soc.} \textbf{61} (2000), 58-80.

\bibitem{Brouwer} Brouwer, A. E., The eigenvalues of oppositeness graphs in spherical buildings, {\it Combinatorics and Graphs}, R.A. Brualdi, S. Hedayat, H. Kharaghani, G.B. Khosrovshahi, S. Shahriari (eds.), AMS Contemporary Mathematics Series 531, 2010.

\bibitem{BCN} Brouwer, A. E,  Cohen, A, Neumaier, A {\it Distance Regular Graphs},  Ergebnisse der Mathematik und ihrer Grenzgebiete 3.18, Springer-Verlag, 1989.

\bibitem{CarterBook} Carter, Roger W., {\it Simple groups of Lie type},
Pure and Applied Mathematics, Vol. 28,
John Wiley \& Sons, London-New York-Sydney, 1972.

\bibitem{CarterLusztig} Carter, R. W. and Lusztig, G., Modular representations of finite groups of Lie type, Proc. London Math. Soc. (3), {\bf 32},
(1976), 347--384.

\bibitem{CSX} D. B. Chandler, P. Sin, Q. Xiang, The permutation action of finite symplectic groups of odd characteristic on their Standard Modules, J. Algebra {\bf 318} (2007), 871--892.

\bibitem{CohCoop} Cohen, A, Cooperstein, The 2-spaces of the standard $E_6$-module, 
Geometriae Dedicata 25 (1988) 467–-480.

\bibitem{Cooperstein} Cooperstein, B On a Connection Between Hyperbolic Ovoids on the Hyperbolic Quadric Q (10,q) and the Lie Incidence Geometry $E_{6,1}(q)$, 55-64 in Groups and Geometries, Birkhauser, 1998

\bibitem{GLS3} D. Gorenstein, R. Lyons, R. Solomon, {\it The Classification of the Finite Simple Groups. Number 3. Part I. Chapter A, Almost Simple $K$-Groups},
Mathematical Surveys and Monographs, vol. 40.3,
American Mathematical Society, Providence, RI, 1998.

\bibitem{Jantzen} Jantzen, J. C.,  {\it Representations of Algebraic Groups}, Academic Press, London, 1987.

 \bibitem{SaS} N. S. N. Sastry and P. Sin,
 The code of a regular generalized quadrangle of even order,
 in {\it Group Representations: Cohomology, Group Actions and Topology},
 Proc. Symposia in Pure Mathematics {\bf 63} (1998), 485--496.

\bibitem{Sin} Sin, P., The $p$-rank of the incidence matrix of intersecting 
linear subspaces, Designs, Codes and Cryptography {\bf 31} (2004), 213--220.

\bibitem{St} R. Steinberg, Representations of algebraic groups, Nagoya Math. J.
{\bf 22} (1963), 33-56.

\bibitem{St2} R. Steinberg, Endomorphisms of linear algebraic groups,
Memoirs Amer. Math. Soc {\bf 80} (1968).

\bibitem{Tits} Tits, Jacques, {\it Buildings of spherical type and finite {BN}-pairs},
Lecture Notes in Mathematics, Vol. 386, Springer-Verlag, Berlin 1974.
\end{thebibliography}
\end{document}